\documentclass[10pt]{amsart}
\usepackage{amsmath}
\usepackage{amssymb}
\usepackage{amsthm}
\usepackage{amsrefs}
\usepackage{dsfont}
\usepackage{mathrsfs}
\usepackage{stmaryrd}
\usepackage[all]{xy}
\usepackage{verbatim}  

\DeclareMathOperator{\id}{id}

\DeclareMathOperator{\im}{im}

\DeclareMathOperator{\Tor}{Tor}

\DeclareMathOperator{\colim}{colim}

\DeclareMathOperator{\hocolim}{hocolim}

\DeclareMathOperator{\ho}{Ho}

\DeclareMathOperator{\op}{op}

\newcommand{\bE}{\mathbf{E}}
\newcommand{\bF}{\mathbf{F}}

\newcommand{\bR}{\mathbf{R}}

\newcommand{\bT}{\mathbf{T}}

\newcommand{\bZ}{\mathbf{Z}}

\newcommand{\bbL}{\mathbb{L}}

\newcommand{\bbR}{\mathbb{R}}

\newcommand{\sK}{\mathscr{K}}

\newcommand{\sS}{\mathscr{S}}

\newcommand{\sU}{\mathscr{U}}

\newcommand{\Si}{\Sigma^{\infty}}

\newcommand{\Om}{\Omega^{\infty}}

\renewcommand{\phi}{\varphi}

\providecommand{\x}{\times}

\providecommand{\sma}{\wedge}

\providecommand{\arr}{\longrightarrow}

\providecommand{\mGL}{\mathrm{GL}}

\providecommand{\Aut}{\mathrm{Aut}}
\providecommand{\End}{\mathrm{End}}
\providecommand {\Mod}{\mathrm{Mod}}

\numberwithin{equation}{section}

\newtheorem{theorem}{Theorem}[section]
\newtheorem{proposition}[theorem]{Proposition}
\newtheorem{lemma}[theorem]{Lemma}

\theoremstyle{definition}
\newtheorem{definition}[theorem]{Definition}
\newtheorem{remark}[theorem]{Remark}
\newtheorem{construction}[theorem]{Construction}

\title[Bundles of spectra and algebraic $K$-theory]{Bundles of spectra and algebraic $K$-theory}
\author[John A. Lind]{John A. Lind}
\address{Fakult\"at f\"ur Mathematik\\Universit\"at Regensburg}
\email{john.alexander.lind@gmail.com}

\begin{document}

\begin{abstract}
A parametrized spectrum $E$ is a family of spectra $E_{x}$ continuously parametrized by the points $x \in X$ of a topological space.  We take the point of view that a parametrized spectrum is a bundle-theoretic geometric object.  When $R$ is a ring spectrum, we consider parametrized $R$-module spectra and show that they give cocycles for the cohomology theory determined by the algebraic $K$-theory $K(R)$ of $R$ in a manner analogous to the description of topological $K$-theory $K^0(X)$ as the Grothendieck group of vector bundles over $X$.  We prove a classification theorem for parametrized spectra, showing that parametrized spectra over $X$ whose fibers are equivalent to a fixed $R$-module $M$ are classified by homotopy classes of maps from $X$ to the classifying space $B \Aut_{R} M$ of the topological monoid of $R$-module equivalences from $M$ to $M$.  
\end{abstract}

\maketitle


\tableofcontents

\section{Introduction}

Contemporary algebraic topology features a vast array of generalized cohomology theories, but our knowledge of their geometric content remains limited to the examples of ordinary cohomology theories, topological $K$-theory and cobordism theories.  In this paper we describe the geometry underlying the cohomology theory associated to the algebraic $K$-theory of a ring, or more generally a ring spectrum.  The higher algebraic $K$-groups $K_n(R)$ of a ring spectrum $R$ may be defined as the homotopy groups of the algebraic $K$-theory spectrum $K(R)$.  By the geometry of $K(R)$-theory, we mean a geometric description of the cocycles whose equivalence classes form the cohomology groups $K(R)^{*}(X)$ associated to the spectrum $K(R)$.  Our methods only give a description of the degree zero cohomology group $K(R)^0(X)$ and the result is reminiscent of the description of topological $K$-theory $K^0(X)$ in terms of the Grothendieck group of vector bundles over $X$.  The analogue of vector bundles for $K(R)$-theory are parametrized spectra that are modules over the ring spectrum $R$.  We call these objects $R$-bundles.  The main result is the following:

\begin{theorem}\label{main_theorem}
Let $R$ be a connective ring spectrum and let $K(R)$ be the algebraic $K$-theory spectrum of $R$.  Then for any finite CW complex $X$, there is a natural isomorphism
\[
K(R)^{0}(X) \cong \mathrm{Gr} [ \text{lifted, free, finite rank $R$-bundles over $X$}]
\]
between the degree zero $K(R)$-cohomology classes of $X$ and the Grothendieck group of the abelian monoid of equivalence classes of lifted $R$-bundles over $X$ that are free and finite rank as parametrized $R$-modules.
\end{theorem}

\noindent We will give a precise meaning to all of the terms occurring in the statement of the theorem in \S \ref{section:bundle_classification} and \S\ref{proving_main_theorem_section}, but for now we note that an $R$-bundle $E$ over $X$ is free of finite rank if every fiber $E_x$ admits an equivalence of $R$-modules to the $n$-fold wedge $R^{\vee n}$ for some $n \geq 0$.  

Our geometric description of $K(R)$-theory is inspired by previous work.  When $R$ is a discrete ring, Karoubi gave a similar description of the cocycles for $K(R)$-theory in terms of fibrations of projective $R$-modules \cite{Kar}.  When $R$ is the connective complex $K$-theory spectrum $ku$, Baas, Dundas, Richter and Rognes interpreted the cocycles of $K(ku)$-theory as 2-vector bundles, which are a categorification of complex vector bundles \citelist{\cite{BDR} \cite{BDRR}}.  

By definition, $K(R)^0(X)$ is the group of homotopy classes of maps from $X$ to the underlying infinite loop space of the algebraic $K$-theory spectrum, whose homotopy type can be described using Quillen's plus construction:
\[
\Omega^{\infty} K(R) \simeq K_0(R) \times B \mGL_{\infty}^{+}(R).
\]
Here $K_0(R) = K_0^{f} (\pi_0 R)$ is the Grothendieck group of free modules over the discrete ring $\pi_0 R$ and $B \mGL_{\infty}^{+}(R)$ is Quillen's plus construction applied to the H-space $B \mGL_{\infty}(R) = \colim_n B\mGL_n(R)$, where $B \mGL_n(R)$ is the classifying space of the derived mapping space $\mGL_n(R) = \mathbf{Aut}_R(R^{\vee n})$ of $R$-module equivalences $R^{\vee n} \arr R^{\vee n}$.  

One important point is that, unlike the case of vector bundles and complex $K$-theory, the plus construction can radically change the homotopy type.  This forces the bundles that define cocycles for $K(R)$-theory to be \emph{lifted} $R$-bundles over $X$, meaning $R$-bundles defined up to covers of $X$ with homologically trivial fibers---see \S\ref{proving_main_theorem_section} for a precise definition.

The term ``bundle'' is perhaps a little naive: as one continuously varies the basepoint in $X$, the fibers of a parametrized spectrum are weak homotopy equivalent, but need not be strictly isomorphic.  Put another way, to describe a parametrized spectrum in terms of cocycle data would require a derived or infinitely homotopy coherent descent condition.  This point of view naturally leads to a description of parametrized objects as homotopy sheaves with values in a quasicategory, as developed by Ando, Blumberg, Gepner, Hopkins and Rezk \citelist{\cite{ABG1} \cite{ABG2} \cite{ABGHR2}}.  Rather than using quasicategories, we follow the foundations of parametrized stable homotopy theory developed by May and Sigurdsson \cite{MS}.  In their framework, parametrized spectra are defined in terms of a ``total object'' over $X$ instead of cocycle data.  Homotopical control of the fiber homotopy type of parametrized spectra is maintained via the framework of Quillen model categories.

Theorem \ref{main_theorem} follows from a general classification theorem for parametrized $R$-module spectra.  In this paper, a spectrum means an orthogonal spectrum, and we use the stable model structure on orthogonal ring and module spectra from Mandell-May-Schwede-Shipley \cite{MMSS}.  Given an $R$-module $M$, we say that a parametrized $R$-module spectrum $E$ over $X$ has fiber $M$ if the fiber $E_{x}$ of $E$ over every point $x \in X$ admits a stable equivalence $E_{x} \simeq M$ of $R$-modules.  We use the terms $R$-bundle with fiber $M$ and parametrized $R$-module with fiber $M$ interchangeably.  Let $\mathbf{Aut}_{R} M$ be the derived mapping space of homotopy automorphisms of $M$ as an $R$-module.  In \S\ref{section:bundle_classification}, we explain how to realize this homotopy type as a group-like topological monoid, so that we may form the classifying space $B \mathbf{Aut}_{R} M$, and prove the following classification theorem.

\begin{theorem}\label{thm:classification_of_Rbundles} Let $X$ be a CW complex, let $R$ be a ring spectrum and let $M$ be an $R$-module.  There is a natural bijection between stable equivalence classes of $R$-bundles over $X$ with fiber $M$ and homotopy classes of maps $[X, B \mathbf{Aut}_{R} M]$.
\end{theorem}

When $M = R$, Theorem \ref{thm:classification_of_Rbundles} says that line $R$-bundles over $X$ are classified by the classifying space $B \mGL_1 R$ of the units of $R$.  The construction of the line $R$-bundle associated to a map $f \colon X \arr B \mGL_1 R$ is the generalized Thom spectrum studied by Ando-Blumberg-Gepner-Hopkins-Rezk \citelist{\cite{ABGHR1} \cite{ABGHR2}}; see Remark \ref{construction_thomspectra}.  From another point of view, a parametrized spectrum with fiber $M$ gives a twisted form of the cohomology theory $M$.  We can then view Theorem \ref{thm:classification_of_Rbundles} as giving a general classification theorem of the twists of $M$-theory.

In their $\infty$-categorical approach to parametrized homotopy theories, Ando-Blumberg-Gepner \cite{ABG2}*{B.4} proved that the quasicategory of morphisms 
$
\Pi_{\infty} X \arr \sS_{\infty}
$
from the singular simplicial complex of a space $X$ to the quasicategory of spectra $\sS_{\infty}$ is equivalent to the quasicategory associated to the May-Sigurdsson model category of parametrized spectra over $X$.  Variants of their arguments can be used to prove results in the same vein as Theorem \ref{thm:classification_of_Rbundles}.  The proof in this paper is more concrete, using the pullback of a universal bundle to induce the equivalence instead of Lurie's straightening functor \cite{HTT}*{\S3.2.1}.

In order to prove Theorem \ref{thm:classification_of_Rbundles}, we compare $R$-bundles with fiber $M$ and principal $\Aut_{R}M$-fibrations, where $\Aut_{R}M$ is a point-set model for the derived mapping space of homotopy automorphisms constructed out of appropriate cofibrant and fibrant approximations.  Much of the technical material in the paper goes into maintaining control of the fiberwise homotopy type of the principal fibration associated to an $R$-bundle with fiber $M$.  By carefully intertwining a Quillen-type model structure and a Hurewicz-type model structure, we show that this construction induces a bijection of equivalence classes, and reduce the proof of the classification theorem for $R$-bundles with fiber $M$ to the classification theorem for principal fibrations.  

The classification theorem for $R$-bundles and the construction of the principal fibration associated to an $R$-bundle has recently been used by Cohen and Jones \citelist{\cite{cohen_jones} \cite{cohen_jonesII}} in their study of the gauge group of parametrized spectra and the $K$-theory of string topology.  


\subsection*{Outline}  In \S\ref{sec:new_principal_fibrations}, we collect the necessary facts about model category structures on parametrized spaces, introduce a homotopical notion of a $G$-torsor and compare it to that of a principal $G$-fibration, where $G$ is a topological monoid.  The model category structures on parametrized spectra are recalled in \S\ref{section:modelcat_spectra}, then in \S\ref{section:prep} we construct the principal $\Aut_{R}M$-fibration associated to a bundle with fiber $M$.  We prove in \S\ref{section:bundle_classification} that this construction provides an inverse up to homotopy to the associated bundle construction
\[
Y \longmapsto M \sma_{\Sigma^{\infty}_{+}\Aut_{R} M} \Sigma^{\infty}_{B} Y,
\]
and prove Theorem \ref{thm:classification_of_Rbundles}.  The proof of Theorem \ref{main_theorem} is given in \S\ref{proving_main_theorem_section}.

\subsection*{Topological Conventions}  We will rely heavily on the foundations for parametrized homotopy theory developed by May-Sigurdsson \cite{MS}.  As explained there, it is advantageous to leave the category $\sU$ of compactly generated spaces.  By a ``space'' we mean a $k$-space as defined in \cite{MS}*{1.1.1}, and we denote the category of spaces by $\sK$.  We will always assume that the base space (denoted by $B$ or $X$) is compactly generated.  We assume throughout that the ring spectrum $R$ is well-grounded, meaning that each constituent space is compactly generated and non-degenerately based.

\subsection*{Acknowledgments}  I thank Peter May for his support and interest in this project, as well as Matt Ando, Andrew Blumberg and Mike Shulman for stimulating conversations.  The advice of an anonymous referee improved the quality of the manuscript, and I thank them for their efforts.  This work was partially supported by the DFG through SFB1085


\section{Model category theory and principal fibrations}\label{sec:new_principal_fibrations}

In this section, we recall some basic material on model category structures on the category of parametrized spaces from May-Sigurdsson \cite{MS}.  We then introduce a homotopical notion of a $G$-torsor, where $G$ is a topological monoid, and show that it is equivalent to that of a principal $G$-fibration.

The category $\sK$ of $k$-spaces admits a compactly generated topological model structure with weak equivalences the weak homotopy equivalences, fibrations the Serre fibrations, and cofibrations the retracts of relative cell complexes.  We refer to this model structure as the $q$-model structure, and use the terms $q$-equivalences, $q$-fibrations, and $q$-cofibrations for its weak equivalences, fibrations, and cofibrations.  Let $B$ be a compactly generated topological space.  The category $\sK/B$ of spaces $(X, p) = (p \colon X \arr B)$ over $B$ admits a model structure whose weak equivalences and fibrations are detected by the forgetful functor $(X, p) \longmapsto X$ to the $q$-model structure on $\sK$.  An ex-space is a space $(X, p)$ over $B$ along with a map $s \colon B \arr X$ such that $p \circ s = \id_{B}$.  The category $\sK_{B}$ of ex-spaces $(X, p, s)$ also admits a model structure given by the forgetful functor to the $q$-model structure on $\sK$.  We refer to these model structures as the $q$-model structure on $\sK/B$ and $\sK_{B}$, respectively.  While both of these model structures are compactly generated and topological, they are not well-grounded, in the sense of \cite{MS}*{\S 5.3-5.6}.  The problem is that the generating $q$-cofibrations and acyclic $q$-cofibrations do not satisfy the homotopy extension property defined in terms of fiberwise or fiberwise pointed homotopies in $\sK/B$ or $\sK_{B}$ \cite{MS}*{5.1.7, 5.1.8}.  Instead, they are only Hurewicz cofibrations in the underlying category of spaces.  As a result, applications of the glueing lemma that would allow standard inductive arguments over cell complexes built out of the generating sets fail for these model structures.  In attempting to construct a stable model structure on parametrized spectra based on the $q$-model structure, the verification that relative cell complexes built out of the generating acyclic cofibrations are weak equivalences is unattainable.

As an alternative, May-Sigurdsson develop the $qf$-model structure on $\sK/B$ and $\sK_{B}$ \cite{MS}*{\S 6.1-6.2}.  The $qf$-model structure also has the $q$-equivalences as weak equivalences, so that the associated homotopy category is still the homotopy category of spaces over $B$, but there are fewer $qf$-cofibrations than $q$-cofibrations.  A $qf$-fibration need not be a Serre fibration but is a quasifibration.  For our purposes, we do not need the details of the definitions, only the fact that in each case the $qf$-model structure is a well-grounded compactly generated model category.  We will work in the un-sectioned context, building well-grounded compactly generated model structures on parametrized diagram spaces out of the $qf$-model structure on $\sK/B$.  

The category of spaces over $B$ is tensored over the category of spaces via the cartesian product
\begin{align*}
\sK \times \sK/B &\longrightarrow \sK/B \\
 (X, Y \overset{p}{\arr} B) &\longmapsto (X \times Y \xrightarrow{p \circ \pi_2} B).
\end{align*} 
If $G$ is a topological monoid, then we use this structure to define the notion of an object of $\sK/B$ with a strictly associative and unital (left) action of $G$, which we call a $G$-space over $B$.  The $G$-spaces over $B$ form a category $G \sK/B$ with morphisms the $G$-equivariant maps over $B$.  Equivalently, the category $G \sK/B$ is the comma category $(G \sK \downarrow B)$ formed in the category of $G$-spaces, where we consider $B$ to have a trivial action of $G$.  We will also need the $qf$-model structure on the category of $G$-spaces over $B$.

\begin{proposition}\label{prop:base_change_quillen}
Let $G$ be a topological monoid.  There is a well-grounded compactly generated model category structure on the category of $G \sK/B$ of $G$-spaces over $B$ with weak equivalences and fibrations created by the forgetful functor to the $qf$-model structure on spaces over $B$.  If $f \colon A \arr B$ is a map of spaces, then the pullback functor $f^* \colon G \sK/B \arr G \sK/A$ and its left adjoint $f_{!}$ form a Quillen adjoint pair for the $qf$-model structure.  If $f$ is a $q$-equivalence of spaces, then $(f_{!}, f^*)$ is a Quillen equivalence.  
\end{proposition}
\begin{proof}
The corresponding statements when $G = \ast$ are 6.2.5, 7.3.4, and 7.3.5 of \cite{MS}.  The generating cofibrations and acyclic cofibrations for the associated model structure on the category of parametrized $G$-spaces are obtained by applying the free $G$-space functor $G \times (-)$, defined in terms of the tensor of a space and a space over $B$, to the generating sets for the $qf$-model structure on $\sK/B$.  The result then follows by directly checking the criteria for compactly generated model structures in \cite{MS}*{5.5.1}.  Note that the $qf$-model structure on $G\sK/B$ inherits the property of being right proper from $\sK/B$, so it is a well-grounded model structure, see \cite{MS}*{5.5.4}.  The claims about the adjunction follow directly from the case $G = \ast$.
\end{proof}

In particular, the fiber functor $i_{b}^* = (-)_{b}$ is right Quillen on the category of $G$-spaces over $B$.  We let $\bF_{b} = \bR i_{b}^*$ denote its right derived functor.  In other words, $\bF_{b} Y$ is the object of the homotopy category of $G$-spaces determined by the fiber $(R^{qf}Y)_{b}$ of a $qf$-fibrant approximation of $Y$.  

While the following terminology is non-standard, it will be useful as an intermediary between the highly structured notion of a principal $G$-fibration and the model-theoretic fiber conditions on parametrized spectra.

\begin{definition}\label{def:Gtorsor}
A $G$-torsor over $B$ is a $G$-space $(Y, p)$ over $B$ for which every derived fiber $\bF_{b}Y$ admits a zig-zag of $q$-equivalences of $G$-spaces to $G$, considered as a $G$-space via left multiplication.  We write $\ho (G \Tor/B)$ for the full subcategory of the homotopy category $\ho (G \sK/B)$ of $G$-spaces over $B$ that is spanned by the $G$-torsors.
\end{definition}

The notion of a $G$-torsor is native to the Quillen model structure.  The following definition instead uses the Hurewicz model structure.
\begin{definition}
A principal $G$-fibration over $B$ is a $G$-space $(Y, p)$ over $B$ for which
\begin{itemize}
\item the structure map $p \colon Y \arr B$ is an $h$-fibration of $G$-spaces, meaning that it has the homotopy lifting property in the category of $G$-spaces.
\item for every $b \in B$, there is a zig-zag of weak equivalences of $G$-spaces $Y_b \simeq G$.
\end{itemize}
\end{definition}

We now construct an approximation functor $\Gamma$ in order to compare $G$-torsors with principal $G$-fibrations.  Given a $G$-space $(Y, p)$ over $B$, let $(\Gamma Y, \Gamma p)$ be the $G$-space over $B$ defined by the mapping path-space construction
\begin{align*}
\Gamma p \colon \Gamma Y =  B^I \times_{B} Y  &\xrightarrow{\mathrm{ev}_1} B \\
(\gamma, y) &\longmapsto \gamma(1),
\end{align*}
and note that the fiber $(\Gamma Y)_b$ of $\Gamma p$ is the homotopy fiber of $Y$ at $b \in B$.
Note that the map $\Gamma p$ is an $h$-fibration of $G$-spaces, since the lifting problem in the category of $G$-spaces
\[
\xymatrix{
X \ar[r]^-{(\gamma, f)} \ar[d]_{i_0} & B^I \times_{B} Y \ar[d]^{\mathrm{ev}_1} \\
X \times I \ar[r]^{h} \ar@{-->}[ur]^{\widetilde{h}} & B
}
\]
has a solution given by $\widetilde{h}(x, t) = (\lambda_{t}(x), f(x))$, where $\lambda_{t}(x)$ is the path
\[
\lambda_{t}(x)(s) = \begin{cases}
\gamma(x)(s + st) \quad & 0 \leq s \leq 1/(1 + t) \\
h(x, s + ts - 1) \quad & 1/(1 + t) \leq s \leq 1,
\end{cases}
\]
and the map $\widetilde{h}$ is evidently $G$-equivariant.

The construction of mapping path-spaces is functorial, so that $\Gamma$ defines an endofunctor of the category of $G$-spaces over $B$ with the following easily verifiable properties. 
\begin{lemma}\label{gamma_properties} \hspace{2in}
\begin{itemize}
\item[(i)]  
If $p$ is a quasifibration and every fiber $Y_b$ is $q$-equivalent to $G$, then $(\Gamma Y, \Gamma p)$ is a principal $G$-fibration over $B$. 
\item[(ii)]  Suppose that $(X, p) \arr (Y, q)$ is a $q$-equivalence of $G$-spaces over $B$.  Then the induced map $(\Gamma X, \Gamma p) \arr (\Gamma Y, \Gamma q)$ is a $q$-equivalence of principal $G$-fibrations.
\item[(iii)]  The map $(Y, p) \arr (\Gamma Y, \Gamma p)$ defined by the inclusion into constant paths is a homotopy equivalence of $G$-spaces over $B$.  If $p$ is a quasifibration, then the map restricts to a $q$-equivalence on fibers.
\end{itemize}
\end{lemma}

\begin{proposition}\label{prop:Gtors_are_Gfibs}
The functor $\Gamma$ induces a natural isomorphism between the set of $q$-equivalence classes of $G$-torsors over $B$ and the set of $q$-equivalence classes of principal $G$-fibrations over $B$.
\end{proposition}

\begin{proof}

Let $(Y, p)$ be a $G$-space over $B$.  The inclusion of the fiber into the homotopy fiber for both $(Y, p)$ and a $qf$-fibrant approximation $(R^{qf}Y, R^{qf}p)$ are related by the commutative diagram
\addtocounter{theorem}{1}
\begin{equation}\label{diagram:justify}
\xymatrix{
Y_{b} \ar[d] \ar[r] & (R^{qf}Y)_{b} \ar[d]^{\simeq} \\
(\Gamma Y)_b \ar[r]^-{\simeq} & (\Gamma R^{qf} Y)_b
}
\end{equation}
induced by fibrant approximation and the inclusion of constant paths.  Since the fibrant approximation is a $q$-equivalence of total spaces, it induces a $q$-equivalence of the homotopy fibers.  The $qf$-fibration $R^{qf}p$ is in particular a quasifibration, which gives the other displayed $q$-equivalence.  It follows that the derived fiber $\bF_{b}Y$ is canonically $q$-equivalent to the homotopy fiber $(\Gamma Y)_b$.  

In particular, $\Gamma$ takes $G$-torsors to principal $G$-fibrations and preserves $q$-equivalences.  Conversely, every principal $G$-fibration 
is a $G$-torsor.  The map $\eta \colon Y \arr \Gamma Y$ in Lemma \ref{gamma_properties}.(iii) is a $q$-equivalence of $G$-spaces, so $\Gamma$ is bijective on $q$-equivalence classes.

\end{proof}

Using the proposition, the next theorem is a restatement of May's classification theorem for principal $G$-fibrations \cite{class_and_fib}*{9.2}.

\begin{theorem}\label{thm:new_classify_Gfibs}
 Let $G$ be a grouplike topological monoid with non-degenerate basepoint and let $B$ be a CW-complex.  Then taking the pullback of $\Gamma EG \arr BG$ along a given map $B \arr BG$ defines a natural bijective correspondence between the set of homotopy classes of maps $[B, BG]$ and the set of equivalence classes of $G$-torsors over $B$.
 \end{theorem}

\section{Model categories of parametrized spectra}\label{section:modelcat_spectra}

We now summarize what we need from the theory of parametrized spectra, following chapters 11 and 12 of May-Sigurdsson \cite{MS}.  A spectrum over $B$ is an orthogonal spectrum in the category of ex-spaces over $B$.  In other words, a parametrized spectrum $X$ consists of an $O(V)$-equivariant ex-space $(X(V), p(V), s(V))$ for each finite dimensional real inner product space $V$, along with compatible $(O(V) \times O(W))$-equivariant structure maps
\[
\sigma \colon X(V) \sma_{B} S^{W}_{B} \arr X(V \oplus W)
\]
over and under $B$.  Here $S^{V}_{B} = r^* S^{V} = S^{V} \times B$ is the trivially twisted ex-space with fiber the one-point compactification $S^{V}$.  The section  of $S_{B}^{V}$ is determined by the basepoint of $S^{V}$.  The smash product $\sma_{B}$ is the fiberwise smash product of ex-spaces.  A map $f \colon X \arr Y$ of spectra over $B$ consists of an equivariant map $f(V) \colon X(V) \arr Y(V)$ of ex-spaces for each indexing space $V$ that are suitably compatible with the structure maps $\sigma$.  For each point $b \in B$, the fiber of $X$ over $b$ is the spectrum $X_{b} = i_{b}^* X$ given by the pullback of $X$ along the inclusion map $i_{b} \colon \{b\} \arr B$.  The fiber spectrum is described level-wise in terms of the fibers of its constituent ex-spaces by the formula $X_{b}(V) = X(V)_{b}$.

The level model structure on the category $\sS_{B}$ of spectra over $B$ has as weak equivalences, respectively fibrations, those maps $f$ such that each $f(V)$ is a $q$-equivalence, respectively $qf$-fibration, of ex-spaces.  We refer to these maps as the level-wise $q$-equivalences and level-wise $qf$-fibrations, respectively.  The homotopy groups of a level-wise $qf$-fibrant spectrum $X$ over $B$ are the homotopy groups $\pi_q X_{b}$ of all of the fibers of $X$.  The homotopy groups of a spectrum $X$ over $B$ are the homotopy groups $\pi_q (R^{l} X)_{b}$ of the fibers of a level-wise $qf$-fibrant approximation $R^{l}X$ of $X$.  We say that a map $X \arr Y$ of spectra over $B$ is a stable equivalence if it induces an isomorphism on all homotopy groups of all fibers.  An $\Omega$-spectrum over $B$ is a level $qf$-fibrant spectrum $X$ over $B$ whose adjoint structure maps
\[
\widetilde{\sigma} \colon X(V) \arr \Omega_{B}^{W} X(V \oplus W)
\]
are $q$-equivalences of ex-spaces over $B$.  

\begin{theorem}\label{s_model_structure_param_spectra}\cite{MS}*{12.3.10}
The category $\sS_{B}$ of spectra over $B$ admits the structure of a well-grounded compactly generated model category whose weak equivalences are the stable equivalences.  The fibrations and cofibrations are called the $s$-fibrations and the $s$-cofibrations, and the $s$-fibrant objects are the $\Omega$-spectra over $B$.  We refer to this model structure as the $s$-model structure (or stable model structure) on $\sS_{B}$.
\end{theorem}

\noindent In the case $B = *$, this coincides with the stable model structure on orthogonal spectra from Mandell-May-Schwede-Shipley \cite{MMSS}.

Parametrized spaces and parametrized spectra are related by suspension spectrum and underlying infinite loop space functors.  If $(Y, p)$ is a space over $B$, the fiberwise suspension spectrum $\Si_{B} Y$ is the spectrum over $B$ defined by 
\[
(\Si_{B}Y) (V) = (Y, p)_{+} \sma_{B} S_B^{V},
\]
where
\[
(Y, p)_{+} = (Y \amalg B, p \amalg \id_{B}, \id_{B})
\]
is the ex-space over $B$ obtained from $(Y, p)$ by adjoining a disjoint section.
The right adjoint $\Om_{B}$ of $\Si_{B}$ is defined by $\Om_{B} X = X(0)$.  By inspecting the definitions, we see that there are natural isomorphisms of fibers $(\Si_{B} Y)_{b} \cong \Si_{+} Y_{b}$ and $(\Om_{B} X)_{b} \cong \Om X_{b}$.  



The category $\sS_{B}$ of spectra over $B$ is enriched and tensored over the category $\sS$ of spectra with tensor the fiberwise smash product $\sma$.
We use this structure to define parametrized module spectra.  Let $R$ be a (non-parametrized) ring spectrum.  We assume, once and for all, that $R$ is well-grounded, meaning that each $R(V)$ is well-based and compactly generated.    An $R$-module over $B$ is a spectrum $N$ over $B$ with an associative and unital map of spectra 
$R \sma N \arr N$
over $B$.

\begin{theorem}\cite{MS}*{14.1.7}
The category $R \Mod_{B}$ of $R$-modules over $B$ is a well-grounded compactly generated model category with weak equivalences and fibrations created by the forgetful functor to $\sS_{B}$.  We refer to this model structure as the $s$-model structure on $R\Mod_{B}$.
\end{theorem}

If $X$ is a space and $Y$ is a space over $B$, then there is a natural isomorphism of parametrized spectra over $B$
\[ 
\Si_{B}(X \times Y) \cong \Si_{+} X \sma \Si_{B} Y
\]
that satisfies the analogues of the associativity and unit diagrams for a monoidal natural transformation.  Similarly, $\Om_{B}$ preserves the monoidal structure up to a lax monoidal transformation, so that if $G$ is a topological monoid, then the adjunction $(\Si_{B}, \Om_{B})$ restricts to give an adjunction between $G$-spaces over $B$ and $\Si_{+} G$-module spectra over $B$.  

\begin{proposition}\label{prop:sus_is_quillen} \hspace{2in}
\begin{itemize}
\item[(i)]  The adjoint pair $(\Sigma^{\infty}_{B}, \Omega^{\infty}_{B})$ is a Quillen adjunction between the $qf$-model structure on spaces over $B$ and the $s$-model structure on spectra over $B$.
\item[(ii)]  Let $G$ be a topological monoid.  The adjoint pair $(\Sigma^{\infty}_{B}, \Omega^{\infty}_{B})$ is a Quillen adjunction between the $qf$-model structure on $G$-spaces over $B$ and the $s$-model structure on $\Sigma^{\infty}_{+} G$-modules over $B$.
\end{itemize}
\end{proposition}
\begin{proof}
In both cases, this follows by examining the effect of $\Si_{B}$ on generating cofibrations and acyclic cofibrations; since the $s$-model structure on $\sS_{B}$ is a left Bousfield localization of the level $qf$-model structure, its generating sets contain all maps of the form $\Sigma^{\infty}_{B} i$, where $i$ runs through the generating sets for the $qf$-model structure on $\sK/B$.
\end{proof}

It is a formal consequence that the left Quillen functor $\Si_{B}$ preserves weak equivalences between cofibrant objects.  However, it will be useful to know that a stronger result is true.

\begin{lemma}\label{suspension_preserves_we_lemma}
The functor $\Si_{B} \colon \sK/B \arr \sS_{B}$ preserves all weak equivalences.
\end{lemma}
\begin{proof}
If $f \colon X \arr Y$ is a weak homotopy equivalence of spaces over $B$, then each map of ex-spaces $f_+ \sma_{B} S^{V}_{B}$ is a weak homotopy equivalence on total spaces.  This means that $\Sigma^{\infty}_{B} f$ is a level-wise weak homotopy equivalence and thus a stable equivalence of parametrized spectra by \cite{MS}*{12.3.5}.
\end{proof}

We will work in the non-parametrized setting for a moment in order to fix notation on some constructions.  Suppose that $R$ and $A$ are ring spectra.  Consider the function spectrum $F^{R}(-, -)$ of $R$-modules.  If $P$ is an $A$-module, $M$ is an $(R, A)$-bimodule and $N$ is an $R$-module, then $F^{R}(M, N)$ is an $A$-module and we have the following adjunction:
\addtocounter{theorem}{1}
\begin{equation}\label{non_param_function_spectrum_adjunction}
\Mod_{R}(M \sma_{A} P, N) \cong \Mod_{A}(P, F^{R}(M, N)).
\end{equation}
It is a consequence of the fact that the category of $R$-modules is a spectrally enriched model category via the function spectra $F^{R}(-,-)$ that if $M'$ is a cofibrant $R$-module, then the functor $F^{R}(M', -)$ preserves stable equivalences between fibrant $R$-modules, and similarly if $N$ is a fibrant $R$-module, then the functor $F^{R}(-, N)$ preserves stable equivalences between cofibrant $R$-modules.

We will be interested in the generalization of the adjunction \eqref{non_param_function_spectrum_adjunction} where $N$ and $P$ are parametrized spectra.  The smash product $M \sma_{A} P$ occurring in the parametrized version of the adjunction is built out of the external smash product $\sma \colon \sS \times \sS_{B} \arr \sS_{B}$, as described in \cite{MS}*{\S14.1}.  In particular, there is never a need to internalize the smash product by taking the pullback $\Delta^*$ of a spectrum over $B \times B$ along the diagonal map.  In this situation, we are able to maintain homotopical control of the smash product.

\begin{lemma}\label{half_parametrized_monoid_axiom} 
Let $i \colon X \arr Y$ be an $s$-cofibration of $R$-modules and let $j \colon Z \arr W$ be an $s$-cofibration of spectra over $B$.  Then the pushout product 
\[
i \boxempty j \colon (Y \sma Z) \cup_{X \sma Z} (X \sma W) \arr Y \sma W
\]
is an $s$-cofibration of $R$-modules over $B$ which is a stable equivalence if either $i$ or $j$ is.
\end{lemma}
\begin{proof}
Using the fact that parametrized spectra and $R$-modules are well-grounded categories, we may induct up the cellular filtration of $i$ and $j$, so it suffices to verify the result when $i$ and $j$ are generating cofibrations or generating acyclic cofibrations.  This follows from the case when $R = S$ \cite{MS}*{12.6.5} because $R \sma (-)$ takes $s$-cofibrations and acyclic $s$-cofibrations of spectra over $B$ to $s$-cofibrations and acyclic $s$-cofibrations of $R$-modules over $B$. 
\end{proof}

The lemma has the following consequence.

\begin{proposition}\label{prop:param_quillen_adjunction}  
Suppose that $M$ is an $(R, A)$-bimodule that is cofibrant as an $R$-module.  Then the adjunction
\[
\xymatrix{
(\text{$A$-modules over $B$}) \ar@<.5ex>[rrr]^-{M \sma_{A} (-) } & & & (\text{$R$-modules over $B$}) \ar@<.5ex>[lll]^-{F^{R}(M, -)}
}
\]
is a Quillen adjunction.
\end{proposition}
\begin{proof}
It follows from the lemma that the adjunction is Quillen when $A = S$.  In particular, the functor $F^{R}(M, -)$ is right Quillen when we consider its codomain to be parametrized spectra.  The general case then holds as well because $s$-fibrations and weak equivalences of $A$-modules over $B$ are created by the forgetful functor to parametrized spectra.  
\end{proof}

\section{The principal $\Aut_{R}M$-fibration associated to an $R$-bundle}\label{section:prep}

Let $R$ be a ring spectrum and let $M$ be an $R$-module.  In this section, we will define the topological monoid $\Aut_{R}M$ of autoequivalences of $R$-modules $M \arr M$.  We then describe the construction of an $\Aut_{R}M$-torsor from an $R$-bundle with fiber $M$.  

Suppose that $G$ is a topological monoid.  While $G$ may not be grouplike, there is a maximal grouplike submonoid $G^{\x} \subset G$ defined as the pullback
\begin{equation}\label{diagram:def_units}
\xymatrix{
G^{\times} \ar[r] \ar[d] & G \ar[d] \\
(\pi_0 G)^{\x} \ar[r] & \pi_0 G }
\end{equation}
where $(\pi_0 G)^{\times} \subset \pi_0 G$ is the subset of invertible elements of the monoid $\pi_0 G$.  In other words, the inclusion $G^{\times} \arr G$ is given by the inclusion of those path components that are invertible under the monoid multiplication.  For example, if $G = \Om R = R(0)$ is the mulitplicative topological monoid underlying an $s$-fibrant ring spectrum $R$, then $G^{\times} = \mGL_1R$ is the space of units of $R$.  A more delicate construction is required if $R$ is commutative and one wants to keep control of the resulting $E_{\infty}$-space structure on $\mGL_1R$ \citelist{ \cite{diagram_spaces} \cite{schlichtkrull_units} \cite{SS_groupcompletion}}, but we will not need this for our purposes.

We assume for the rest of the section that $R$ is an $s$-cofibrant ring spectrum and that $M$ is an $s$-fibrant and $s$-cofibrant $R$-module.  The function spectrum $F^{R}(M, M)$ is a ring spectrum under composition of maps and our assumptions guarantee that it is $s$-fibrant.  Let 
\[
\End_{R}M = \Om F^{R}(M, M) = F^{R}(M, M)(0)
\]
be the underlying topological monoid.  We define $\Aut_{R} M$ to be the units of the ring spectrum $F^{R}(M, M)$:
\[
\Aut_{R} M = \mGL_{1} F^{R}(M, M) = (\Omega^{\infty} F^{R}(M, M))^{\times}
\]
We think of $\Aut_{R} M$ as the space of weak equivalences of $R$-modules $M \arr M$, with monoid multiplication given by composition.  The suspension spectrum of the monoid $\Aut_{R} M$ is a ring spectrum $\Si_{+} \Aut_{R} M$.  The $R$-module $M$ also has the structure of a right $\Si_{+} \Aut_{R} M$-module, with action map
\[
M \sma_{S} \Si_{+} \Aut_{R} M \arr M
\]
the adjoint of the composite map of ring spectra
\[
\Si_+ \Aut_{R} M \arr \Si_{+} \Omega^{\infty} F^{R}(M, M) \overset{\epsilon}{\arr} F^{R}(M, M)
\]
induced by the canonical inclusion $\mGL_1 \arr \Om$ and the counit of the adjunction $(\Si_{+}, \Om$).  Thus $M$ is a $(R, \Si_{+} \Aut_{R} M)$-bimodule.  




We write $\bF_{b} = \bbR i_{b}^* (-)$ for the right derived fiber functor.  If $N$ is an $R$-module over $B$, the derived fiber $\bF_{b} R$ is the object of the homotopy category of $R$-modules determined by the fiber $i_{b}^* R^{s}N$ of an $s$-fibrant approximation of $N$ as an $R$-module over $B$.

\begin{definition}\label{def:Rbundle}
An $R$-bundle over $B$ with fiber $M$ is an $R$-module $N$ over $B$ such that every derived fiber $\bF_{b}N$ of $N$ admits a zig-zag of stable equivalences of $R$-modules to $M$.  
\end{definition}

Let $N$ be an $R$-bundle over $B$.  The function spectrum $F^{R}(M, N)$ is a $\Si_{+}\End_{R} M$-module over $B$.  Applying $\Omega^{\infty}_{B}$, we get a $\End_{R}M$-space $\Omega^{\infty}_{B}F^{R}(M, N)$ over $B$ which is $qf$-fibrant when $N$ is $s$-fibrant.  The following lemma allows us to keep control of its fiber homotopy type.  It is a direct consequence of the cofibrancy of $M$ as an $R$-module.

\begin{lemma}\label{choice_of_equiv_lemma}  Suppose that $N$ is $s$-fibrant and fix a point $b \in B$.  A stable equivalence of $R$-modules $N_{b} \simeq M$ determines:
\begin{itemize}
\item[(i)] a stable equivalence of $\Si_{+}\End_{R}M$-modules $F^{R}(M, N)_{b} \simeq F^{R}(M, M)$, and
\item[(ii)] a $q$-equivalence of $\End_{R}M$-spaces $\Om_{B} F^{R}(M, N)_{b} \simeq \Om F^{R}(M, M).$
\end{itemize}
\end{lemma}



\noindent Notice that the second condition in the lemma implies that $\Om_{B} F^{R}(M, N)$ is an $\End_{R}M$-torsor.  We will now construct an $\Aut_{R}M$-torsor 
\[
E^{R}(M, N) \subset\Om_{B} F^{R}(M, N).
\]
The idea of the construction is to restrict to the subspace whose fiber over $b \in B$ consists of the stable equivalences of $R$-modules $M \arr N_{b}$.  To make this idea rigorous, we need to access the components $\pi_0 \Om_{B} F^{R}(M, N)_{b}$ of each fiber in a way that remembers the topology of $B$.  

To this end, we define the parametrized components $\pi_0^{B} X$ of a parametrized space $p \colon X \arr B$.  As a set, $\pi_0^{B} X$ consists of all components of all fibers of $X$:
\[
\pi_0^{B} X = \bigcup_{b \in B} \pi_0 X_b.
\]
Give $\pi_0^{B} X$ the quotient topology induced by the map $X \arr \pi_0^{B} X$ that sends a point $x \in X$ to its component $[x] \in \pi_0 X_{p(x)}$.  Since the quotient map is a map over $B$, the space $\pi_0^{B} X$ is a parametrized space over $B$.

\begin{construction}\label{def:subtorsor}  We now define a fiberwise version of \eqref{diagram:def_units}.  Let $G$ be a topological monoid and let $(Y, p)$ be a $G$-torsor over $B$ whose structure map $p \colon Y \arr B$ is a quasifibration.  A choice of $q$-equivalence of $G$-spaces $Y_{b} \simeq G$ gives an isomorphism of $\pi_0 G$-spaces $\pi_0 Y_{b} \cong \pi_0 G$.  Define $\pi_0 Y_b^{\x}$ to be the subset of $\pi_0 Y_b$ corresponding to $\pi_0 G^{\x}$ under this isomorphism.  Although the isomorphism $\pi_0 Y_b^{\x} \cong \pi_0 G^{\x}$ of $\pi_0 G^{\x}$-spaces depends on the choice of $q$-equivalence $Y_{b} \simeq G$, the subset $\pi_0 Y_{b}^{\x}$ does not.  Let $\pi_0^{B} Y^{\x} \subset \pi_0^{B} Y$ be the subspace consisting of the sets $\pi_0 Y_{b}^{\x}$ in each fiber.  Define the space $Y^{\x}$ over $B$ to be the following pullback:
\[
\xymatrix{
Y^{\x} \ar[r]^{\iota} \ar[d] & Y \ar[d] \\
\pi_0^{B} Y^{\x} \ar[r] & \pi_0^{B} Y }
\]
Notice that there is a canonical isomorphism $(Y^{\times})_{b} \cong Y_{b}^{\times}$, and that a map $X \arr Y$ of spaces over $B$ factors through $Y^{\times}$ if and only if for every $b \in B$, the induced map $\pi_0 X_b \arr \pi_0 Y_b$ has image lying in $\pi_0 Y_{b}^{\times}$.

 It is straightforward to verify that the construction $Y \mapsto Y^{\times}$ is functorial for maps of $G$-spaces.  We will at times write $\mu = (-)^{\times}$ for the resulting functor.  The assumption that $p$ is a quasifibration is the minimal hypothesis necessary for the construction to be possible.  In practice, $p$ will be either a $qf$-fibration or an $h$-fibration.  
\end{construction}

\begin{lemma}\label{lemma:subtorsor_path_comp}  Suppose that the base space $B$ is semi-locally contractible and that $(Y, p)$ is a principal $G$-fibration over $B$.  Then $\iota \colon Y^{\times} \arr Y$ is the inclusion of a subspace of path components.
\end{lemma}
\begin{proof}
Let $\gamma$ be a path in $Y$ with $\gamma(0) \in Y^{\times}$.  Assuming that $\gamma(1) \notin Y^{\times}$, let $t_0 = \inf \{ t \in [0, 1] \mid \gamma(t) \notin Y^{\times} \}$.  Set $b_0 = p(\gamma(t_0))$ and choose an open neighborhood $U$ of $b_0$ along with a nullhomotopy $h \colon U \times I \arr B$ of $U$ in $B$.  Consider the $G$-space $h^*Y$ over $U \times I$ obtained from $Y$ by pullback along $h$.  The restriction $h^*Y \vert_{U \times \{0\}}$ is isomorphic to $Y \vert_{U}$, while the restriction $h^*Y \vert_{U \times\{1\}}$ is isomorphic to $U \times Y_{b_{0}}$.  It follows that we may find a fiberwise homotopy equivalence of $G$-spaces $\rho \colon Y \vert_{U} \arr U \times Y_{b_{0}}$ over $U$.  Applying the functor $(-)^{\times}$ to $\rho$, we have a commutative diagram
\[
\xymatrix{
Y^{\times}\vert_{U} \ar[r]^-{\rho^{\times}} \ar[d] & U \times Y_{b_0}^{\times} \ar[d] \\
Y \vert_{U} \ar[r]^-{\rho}          & U \times Y_{b_0}
}
\]
which shows that in a neighborhood of $t_0$, the path $\rho \circ \gamma$ must lie in $U \times Y_{b_0}^{\times}$.  Since $\rho$ is a fiberwise homotopy equivalence, it follows that $\gamma(t) \in Y_{p(\gamma(t))}^{\times}$ for $t$ near $t_0$, contradicting our initial assumption.
\end{proof}

\begin{proposition}\label{prop:subtorsor_properties}
Suppose that $(Y, p)$ is a $G$-torsor over a semi-locally contractible space $B$.
\begin{itemize}
\item[(i)] The space $Y^{\times}$ is a $G^{\times}$-space over $B$ and the canonical inclusion $\iota \colon Y^{\times} \arr Y$ is a map of $G^{\times}$-spaces.
\item[(ii)] If the structure map $p \colon Y \arr B$ is an $h$-fibration of $G$-spaces, so that $Y$ is a principal $G$-fibration, then $p^{\times} \colon Y^{\times} \arr B$ is a $G^{\times}$-torsor
\item[(iii)] The functor $\mu \colon Y \mapsto Y^{\times}$ preserves $q$-equivalences between principal $G$-fibrations.
\end{itemize}
\end{proposition}
\begin{proof}
Claim (i) is immediate from the definitions.  For (ii), observe that by Lemma \ref{lemma:subtorsor_path_comp}, $p^{\times}$ is an $h$-fibration of spaces.  It follows that the natural map $Y^{\times}_{b} \arr F_{b}Y^{\times}$ from the fiber to the homotopy fiber is a $q$-equivalence.  A given chain of $q$-equivalences of $G$-spaces $Y_{b} \simeq G$ induces a chain of $q$-equivalences of $G^{\times}$-spaces $Y_{b}^{\times} \simeq G^{\times}$, so we conclude that $Y^{\times}$ is a $G^{\times}$-torsor.


For (iii), assume that $(Y, p) \arr (Z, q)$ is a $q$-equivalence of $G$-torsors with $p$ and $q$ both $h$-fibrations of $G$-spaces.  For any $b \in B$, the induced map of fibers $Y_{b} \arr Z_{b}$ is a $q$-equivalence of $G$-spaces, and so the induced map of $G^{\times}$-spaces $Y^{\times}_{b} \arr Z^{\times}_{b}$ is a $q$-equivalence.  Since $p^{\times}$ and $q^{\times}$ are $h$-fibrations, it follows that $Y^{\times} \arr Z^{\times}$ is a $q$-equivalence on total spaces. 
\end{proof}


As a consequence of Proposition \ref{prop:subtorsor_properties}, we may define the derived functor of $\mu$ to be the functor from the homotopy category of $G$-torsors to the homotopy category of $G^{\times}$-torsors 
\begin{align*}
\boldsymbol{\mu} \colon \ho (G \Tor/B) &\arr  \ho (G^{\times} \Tor/B) \\ 
Y &\longmapsto \mu(\Gamma Y) = (\Gamma Y)^{\times},
\end{align*} 
where $\Gamma$ is the $h$-fibrant approximation functor from \S\ref{sec:new_principal_fibrations}.  Lemma \ref{lemma:subtorsor_path_comp} implies that when $p \colon Y \arr B$ is an $h$-fibration, the fiber $Y_{b}^{\times} \cong (Y^{\times})_{b}$ represents the derived fiber $\bF_{b} Y^{\times}$ of $Y^{\times}$.  In other words: 
\begin{lemma}\label{lemma:mu_commutes}
There is a canonical isomorphism of derived functors $\bF_{b} \boldsymbol{\mu} \cong \boldsymbol{\mu} \bF_{b}$.
\end{lemma}

We will also need to know how to construct maps into $\boldsymbol{\mu}$.

\begin{lemma}\label{lem:factor_through_mu}
A morphism $X \arr Y$ in the homotopy category of $G^{\times}$-spaces over $B$ factors through $\iota \colon \boldsymbol{\mu} Y \arr Y$ if and only if for every $b \in B$, the induced map $\pi_0 \bF_b X \arr \pi_0 \bF_b Y$ has image contained in the subset $\pi_0 \boldsymbol{\mu} \bF_b Y$.

\end{lemma}
\begin{proof}
First notice that the functor $\pi_0 \bF_b$ is invariant under weak equivalences of spaces over $B$.  We may represent a map in the homotopy category of $G^{\times}$-spaces over $B$ by a zig-zag of map where the wrong way maps are weak equivalences, and we assume without loss of generality that the final object in this zig-zag is a $h$-fibrant $G^{\times}$-space over $B$.  The result then follows by using the universal mapping property of the pullback of spaces defining $\mu$.
\end{proof}

\begin{definition}\label{def:associated_torsor}  Let $N$ be an $R$-bundle with fiber $M$ and let $R^{s}N$ be an $s$-fibrant approximation of $N$ as an $R$-module over $B$.  Since $M$ is an $s$-cofibrant $R$-module, the $\End_{R} M$-torsor $\Om_{B} F^{R}(M, R^{s}N)$ is $qf$-fibrant as an $\End_{R}M$-module.  Applying Construction \ref{def:subtorsor} defines an $\Aut_{R}M$-space over $B$
\[
\quad E^{R}(M, R^{s}N) = (\Om_{B} F^{R}(M, R^{s}N))^{\x}
\]
which need not be an $\Aut_{R}M$-torsor.  If we instead take the derived functor $\boldsymbol{\mu}$ by applying the $h$-fibration approximation functor $\Gamma$ before $(-)^{\times}$, then the value of the associated derived functor
\[
\bE^{R}(M, N) = \boldsymbol{\mu}\Om_{B} F^{R}(M, R^{s}N)
\]
is our definition of the $\Aut_{R}M$-torsor associated to the $R$-bundle $N$.  
Since $\Om_{B}$ and $F^{R}(M, -)$ are both right Quillen functors, we can summarize the definition by saying that 
\[
\bE = \bE^{R}(M, -) \colon \ho (\text{$R$-bundles with fiber $M$}) \arr \ho (\text{$\Aut_{R}M$-torsors})
\]
is the composite derived functor $\bE = \boldsymbol{ \mu} \circ \boldsymbol{ \Omega}$, where $\boldsymbol{ \Omega}$ is the right derived functor of $\Omega =  \Om_{B} F^{R}(M, - )$.  
\end{definition}

\section{The classification of $R$-bundles}\label{section:bundle_classification}

In the previous section we constructed an $\Aut_{R}M$-torsor from an $R$-bundle with fiber $M$.  We now construct an $R$-bundle with fiber $M$ from an $\Aut_{R}M$-torsor and show that the constructions are homotopy inverse to each other.  At the end of the section, we complete the proof of Theorem \ref{thm:classification_of_Rbundles}.  We assume that $B$ is a CW-complex, in particular semi-locally contractible, so the functor $\mu$ from the previous section is well-behaved.  We continue to assume that $R$ is an $s$-cofibrant ring spectrum and that $M$ is an $s$-bifibrant $R$-module.

For technical reasons, it will be useful to work with a $q$-cofibrant approximation $\Aut_{R}^{c}M \arr \Aut_{R} M$ of $\Aut_{R} M$ as a topological monoid.  By pullback along the approximation map, any $\Aut_{R}M$-torsor is also an $\Aut_{R}^{c}M$-torsor, so we consider the functor $\bE = \boldsymbol{ \mu} \circ \boldsymbol{ \Omega}$ from the previous section as taking values in $\Aut_{R}^{c}M$-torsors.  Similarly, the right $\Si_{+} \Aut_{R}M$-module structure of $M$ pulls back to give a right $\Si_{+} \Aut_{R}^{c}M$-module structure on $M$.

\begin{definition}\label{def:associated_bundle}
If $Y$ is an $\Aut_{R}^{c}M$-space over $B$, then the fiberwise suspension spectrum $\Si_{B} Y$ is a $\Si_{+} \Aut_{R}^{c}M$-module spectrum over $B$.  The construction
\[
T(Y) = M \sma_{\Si_{+}\Aut_{R}^{c}M} \Si_{B} Y
\]
defines a functor from $\Aut_{R}^{c}M$-spaces over $B$ to $R$-module spectra over $B$ which is left Quillen by Prop. \ref{prop:sus_is_quillen} and \ref{prop:param_quillen_adjunction}.  We let $\bT = \bbL T$ denote its left derived functor.  Note that $\bT$ is left adjoint to the right derived functor $\boldsymbol{ \Omega } = \Om_{B} \bF^{R}(M, -)$. In Proposition \ref{prop:derived_commute}, we will prove that when $Y$ is an $\Aut_{R}^{c}M$-torsor over $B$, then $\bT Y$ is an $R$-bundle with fiber $M$, so that we have a functor
\[
\bT \colon \ho (\text{$\Aut_{R}^{c}M$-torsors})  \arr \ho (\text{$R$-bundles with fiber $M$}).
\]
\end{definition}


\begin{remark}\label{construction_thomspectra} 
In the case $M = R$, the definition recovers the construction of generalized Thom spectra from \citelist{\cite{ABGHR1} \cite{ABGHR2}}.  Given a map of spaces $f \colon B \arr B \mGL_1 R$, the classification of principal $\mGL_1 R$-fibrations gives a principal $\mGL_1 R$-fibration $Y_f$ over $B$.  Applying the functor $T$ then gives a rank one $R$-bundle over $B$.  The Thom spectrum associated to the map $f$ is the (non-parametrized) $R$-module spectrum
\[
\mathrm{M}f= r_{!} T Y_{f} \cong  R \sma_{\Si_{+} \mGL_1^{c} R} \Si_{+} Y_f ,
\]
where $r_{!} \colon \sS_{B} \arr \sS$ is left adjoint to the pullback functor $r^* \colon \sS \arr \sS_{B}$.
\end{remark}

The fiber functor $(-)_b = i_{b}^*$ is a left adjoint, but is not left Quillen for either the stable model structure on parametrized spectra or the $qf$-model structure on parametrized spaces.  However, $i_{b}^*$ is a right Quillen functor.  On the other hand, $T = M \sma_{\Si_{+} \Aut_{R}^{c}M} \Si_{B}(-)$ is a left Quillen functor. There is a natural isomorphism of functors
\[
(M \sma_{\Si_{+} \Aut_{R}^{c} M} \Si_{B} Y)_{b} \cong M \sma_{\Si_{+} \Aut_{R}^{c} M} \Si_{+} Y_{b}
\]
at the point-set level, but this does not imply an isomorphism of derived functors after passage to homotopy categories because we are composing left and right derived functors.  

In order to prove the commutation of derived functors, we will make a slight modification to the functor $T$.  By identifying $R \sma_{S} (\Si_{+} \Aut_{R}^{c} M)^{\op}$-modules with $(R, \Si_{+} \Aut_{R}^{c} M)$-bimodules, the category of $(R, \Si_{+} \Aut_{R}^{c} M)$-bimodules is a well-grounded compactly generated model category with weak equivalences and fibrations created in the $s$-model structure on spectra \cite{MMSS}*{12.1}.  Let $M^{\circ} \arr M$ be an $s$-cofibrant approximation of $M$ as an $(R, \Si_{+} \Aut_{R}^{c} M)$-bimodule and define
\[
T^{\circ}(Y) = M^{\circ} \sma_{\Si_{+}\Aut_{R}^{c}M} \Si_{B} Y
\]  
Note that since $\Si_{+}$ is left Quillen, $\Si_{+} \Aut_{R}^{c}M$ is $s$-cofibrant as a ring spectrum, and thus $s$-cofibrant as a spectrum.  We record a basic consequence.

\begin{lemma}\label{underling_rightmodule_ofM_cofibrant}
The underlying left $R$-module of $M^{\circ}$ is $s$-cofibrant.  The underlying right $\Si_{+}\Aut_{R}^{c}M$-module of $M^{\circ}$ is $s$-cofibrant.
\end{lemma}
\begin{proof}
The right adjoint of the forgetful functor from $(R, \Si_{+}\Aut_{R}^{c} M)$-bimodules to left $R$-modules is the function spectrum functor $F^{S}(\Si_{+} \Aut_{R}^{c}M, -)$.  This functor preserves fibrations and acyclic fibrations because $\Si_{+} \Aut_{R}^{c} M$ is $s$-cofibrant.  Therefore its left adjoint the forgetful functor preserves cofibrations and acyclic cofibrations.  This proves the first claim.  The second claim follows using a similar argument and the fact that $R$ is $s$-cofibrant.
\end{proof}

When $B = \ast$, the fact that $M^{\circ}$ is $s$-cofibrant as a $\Si_{+}\Aut_{R}^{c}M$-module implies that the functor $M^{\circ} \sma_{\Si_{+} \Aut_{R}^{c}M} (-)$ preserves stable equivalences \cite{MMSS}*{12.7}.  Along with Lemma \ref{suspension_preserves_we_lemma}, this shows that the functor $T^{\circ}$ takes $q$-equivalences to stable equivalences when the base is a point.  
The proof of the next result is inspired by Shulman's examples in \S9 of \cite{shulman_doubles}.  To improve clarity, we temporarily revert to the usual notation $\bbL$ and $\bbR$ for left and right derived functors.

\begin{lemma}\label{lemma:derived_commute}  Let $f \colon \ast \arr B$ be the inclusion of a point.  Then there is a natural isomorphism of derived functors $\bbR f^* \bbL T \cong \bbL T \bbR f^*$.
\end{lemma}

\begin{proof} 
The equivalence $M^{\circ} \arr M$ induces an isomorphism of derived functors $\bbR F^{R}(M, -) \cong \bbR F^{R}(M^{\circ}, -)$ since $M$ and $M^{\circ}$ are cofibrant $R$-modules.  This determines an isomorphism of derived functors $\bbL T^{\circ} \cong \bbL T$, so it suffices to prove the result with $T$ replaced by $T^{\circ}$.

Suppose that $X$ is a $qf$-bifibrant $\Aut_R^{c}M$-space over $B$, and consider the following natural transformation of $R$-modules
\begin{equation}\label{commutation_transformation}
T^{\circ} Q^{qf} f^* X \arr T^{\circ} f^* X \overset{\cong}{\arr} f^* T^{\circ} X \arr f^* R^{s} T^{\circ} X,
\end{equation}
where the first and third maps are induced by $qf$-cofibrant approximation and $s$-fibrant approximation, respectively.  Since $T^{\circ}$ preserves all weak equivalences when the base is a point, the first map is a stable equivalence.  The second map is the canonical isomorphism.  It remains to show that $f^*$ preserves the stable equivalence $T^{\circ}X \arr R^{s} T^{\circ}X$.  

Factor $f$ as a $q$-equivalence followed by a $q$-fibration, and consider the two cases separately.  In the first case, the Quillen adjunction $(f_{!}, f^*)$ is a Quillen equivalence both for parametrized $\Aut_{R}^{c}M$-spaces (Prop. \ref{prop:base_change_quillen}) and parametrized $R$-modules (the case of $R = S$ is \cite{MS}*{12.6.7} and the general case follows since stable equivalences and $s$-fibrations of $R$-modules are detected by the forgetful functor to parametrized spectra).  It follows that the natural transformation of derived functors 
\[
\bbL T^{\circ} \bbR f^*  \overset{\eta}{\arr} \bbR f^* \bbL f_{!}  \bbL T^{\circ} \bbR f^* \cong \bbR f^* \bbL T^{\circ} \bbL f_{!} \bbR f^* \overset{\epsilon}{\arr} \bbR f^* \bbL T^{\circ}
\]
is an isomorphism.  As discussed in \cite{shulman_doubles}*{\S7}, this isomorphism of derived functors is represented by the composite \eqref{commutation_transformation}.  In particular, $f^* T^{\circ}X \arr f^* R^{s} T^{\circ}X$ is a stable equivalence in this case, since the map $f$ is still the inclusion of a point.

When $f$ is a $q$-fibration, we instead consider a level-wise $qf$-fibrant approximation $T^{\circ}X \arr R^{l} T^{\circ}X$.  There is a stable equivalence $R^{l} T^{\circ}X \arr R^{s} T^{\circ}X$ under $T^{\circ}X$ \cite{MS}*{12.6.1} and the induced map $f^* R^{l} T^{\circ}X \arr f^* R^{s} T^{\circ}X$ is a stable equivalence because $f^*$ preserves stable equivalences between level-wise $qf$-fibrant spectra.  Pullback along $q$-fibrations preserves weak homotopy equivalences of topological spaces, so $f^* T^{\circ}X \arr f^* R^{l} T^{\circ}X$ is a level-wise $q$-equivalence, hence a stable equivalence.  Therefore $f^* T^{\circ}X \arr f^* R^{s} T^{\circ}X$ is also a stable equivalence.
\end{proof}

We return to using bold-face letters to denote derived functors: $\bT$ is the left derived functor of $T$ and $\bF_{b} = \bbR i_{b}^*$ is the right derived fiber functor.  Recall that the $\Aut_{R}^{c}M$-torsor associated to an $R$-bundle with fiber $M$ is given by
\[
\bE = \bE^{R}(M, -) = \boldsymbol{\mu} \circ \boldsymbol{\Omega},
\]
where $\boldsymbol{\mu}$ is the derived functor of Construction \ref{def:subtorsor} and $\boldsymbol{\Omega}$ is the right derived functor of $\Omega = \Om_{B}F^{R}(M, - )$.

\begin{proposition}\label{prop:derived_commute}
There are natural isomorphisms of derived functors $\bF_{b} \bT \cong \bT \bF_{b}$ and $\bF_{b} \bE \cong \bE \bF_{b}$.
\end{proposition}

\begin{proof}
The first isomorphism is Lemma \ref{lemma:derived_commute}.  For the second, observe that the canonical isomorphism $i_{b}^* \Omega \cong \Omega i_{b}^*$ descends to a canonical isomorphism of derived functors $\bF_{b} \boldsymbol{\Omega} \cong \boldsymbol{\Omega} \bF_{b}$ because $i_{b}^*$ and $\Omega$ are both right Quillen.  By Lemma, \ref{lemma:mu_commutes}, there is a natural isomorphism $\bF_{b} \boldsymbol{\mu} \cong \boldsymbol{\mu} \bF_{b}$, completing the proof.
\end{proof}


In particular, the derived functor $\bT$ takes $\Aut_{R}^{c}M$-torsors to $R$-bundles with fiber $M$, as promised in Definition \ref{def:associated_bundle}.  We are now ready to prove the main theorem of this section.

\begin{theorem}\label{torsor_bundle_equiv_theorem}
The pair of functors $(\bT, \bE)$ defines a bijection between the set of $q$-equivalence classes of $\Aut_{R}^{c} M$-torsors over $B$ and the set of stable equivalence classes of $R$-bundles with fiber $M$ over $B$.
\end{theorem}

\begin{proof}  We work in the homotopy categories of $\Aut_R^{c}M$-spaces over $B$ and $R$-modules over $B$.
Suppose that $Y$ is an $\Aut_{R}^{c}M$-torsor over $B$.  We will construct a natural transformation of derived functors $\zeta \colon Y \arr \bE \bT Y$ by showing that the unit of the adjunction $(\bT, \boldsymbol{ \Omega })$ factors through $\bE^{R}(M, \bT Y)$ as indicated in the following diagram.
\begin{equation}\label{diagram:wanted_factorization}
\xymatrix{
Y \ar[r]^-{\eta} \ar@{-->}[drr]_-{\zeta} & \Om_{B} \Si_{B} Y \ar[r]^-{\eta}  & \Om_{B} \bF^{R}(M, \bT Y ) \\
& & 
\bE^{R}(M, \bT Y) \ar[u]_{\iota}\\
}
\end{equation}
By Lemma \ref{lem:factor_through_mu}, it suffices to prove that if we apply $\pi_0 \bF_b$, then the unit map has image lying in the subset $\pi_0 \boldsymbol{\mu} \bF_b \boldsymbol{\Omega} \bT Y$.


Apply $\bF_{b}$ to diagram \eqref{diagram:wanted_factorization} and commute $\bF_{b}$ past the constituent functors to the input variable $Y$.  Now fix an isomorphism in the derived category $\bF_{b} Y \cong \Aut_{R}^{c}M$ and consider the isomorphic diagram with $\bF_{b} Y$ replaced by $\Aut_{R}^{c}M$.  The composite of the two instances of $\eta$ in this new diagram is the left vertical composite in the following commutative diagram.
\[
\xymatrix{
\Aut_{R}^{c}M \ar[d] \ar[r] & \End_{R} M \ar[d] \\
\Om \Si_{+} \Aut_{R}^{c}M \ar[d] \ar[r]  &  \Om \Si_{+} \End_{R} M \ar[d] \\
 \Om \bF^{R}(M, M \sma_{\Si_{+} \Aut_{R}^{c}M} \Si_{+} \Aut_{R}^{c} M) \ar[r] \ar[dr]^{\cong} &  \Om \bF^{R}(M, M \sma_{\Si_{+} \Aut_{R}^{c}M} \Si_{+} \End_{R} M)  \ar[d] \\
 & \Om \bF^{R}(M, M) 
}
\]
Here the horizontal maps are induced by the composite 
\[
\Aut_{R}^{c}M \arr \Aut_{R} M \arr \End_{R} M
\]
of the cofibrant approximation map and the canonical inclusion.  The diagonal map is induced by the action map
for the right $\Si_{+} \Aut_{R}^{c}M$-module structure on $M$ and it is an isomorphism as indicated.  Since $M$ is bifibrant, we may choose to represent the value of the derived functor $\Om \bF^{R}(M, M)$ in the homotopy category by $\End_{R}M$.  A diagram chase involving the triangle identities for the adjunctions shows that the right vertical composite is then the identity map.  It follows that the left vertical composite factors through $\Aut_{R} M$ via the cofibrant approximation map, and so the map of components 
\[
\pi_0 \bF_b Y \arr \pi_0 \bF_b \Omega^{\infty}_{B} \bF^{R}(M, \bT Y)
\]
lands in the subset $\pi_0 \boldsymbol{\mu} \bF_b \Omega^{\infty}_{B} \bF^{R}(M, \bT Y)$.  This establishes the factorization in diagram \eqref{diagram:wanted_factorization}, and so we have constructed the natural transformation $\zeta \colon Y \arr \bE \bT Y$.

As a consequence of the preceding argument, we see that up to natural isomorphisms in the domain and codomain, $\bF_{b} \zeta$ may be identified with the cofibrant approximation map $\Aut_{R}^{c} M \arr \Aut_{R} M$.  It follows that $\zeta$ is a natural isomorphism of derived functors.

Now let $N$ be an $R$-bundle with fiber $M$.  Define $\xi \colon \bT \bE N \arr N$ to be the composite
\[
\bT \bE^{R}(M, N) \overset{\iota}{\arr} \bT \Om_{B} \bF^{R}(M, N) \overset{\epsilon}{\arr} N
\]
of the map induced by the inclusion $\iota \colon \bE^{R}(M, N) \arr \Om_{B} \bF^{R}(M, N)$ followed by the counit of the adjunction $(\bT, \boldsymbol{ \Omega})$.  After applying the derived fiber functor $\bF_{b}$, commuting it through to the variable $N$, and using a chosen equivalence $\bF_{b} N \simeq M$, an argument similar to that just given for $\zeta$ proves that $\bF_{b} \xi$ is a fiberwise equivalence.  Hence $\xi$ also induces a natural isomorphism of derived functors.

\end{proof}

\begin{proof}[Proof of Theorem \ref{thm:classification_of_Rbundles}]
We return to the general situation of a well-grounded ring spectrum $R$ and an $R$-module $M$.  Take an $s$-cofibrant approximation $R'$ of $R$ as a ring spectrum and an $s$-bifibrant approximation $M'$ of $M$ as an $R$-module so that the material in the last two sections applies.  The derived mapping space $\mathbf{Aut}_{R}M$ of homotopy automorphisms of $M$ has a point-set model given by the space $\Aut_{R'}^{c}M'$.  Theorem \ref{thm:new_classify_Gfibs} and Theorem \ref{torsor_bundle_equiv_theorem} combine to give that homotopy classes of maps $[X, B \mathbf{Aut}_{R}M]$ are in bijective correspondence with equivalence classes of $R'$-bundles with fiber $M'$.  The homotopy category of parametrized $R$-modules and the homotopy category of parametrized $R'$-modules are equivalent by pullback along the approximation map \cite{MS}*{14.1.9}, and the definition of an $R$-bundle with fiber $M$ is invariant under stable equivalences in the entry $M$, so it follows that equivalence classes of $R$-bundles with fiber $M$ are in bijective correspondence with equivalence classes of $R'$-bundles with fiber $M'$.  This completes the proof.
\end{proof}

\section{Lifted $R$-bundles and algebraic $K$-theory}\label{proving_main_theorem_section}

In this section we will prove Theorem \ref{main_theorem}.  
The arguments are adapted from \citelist{\cite{Kar} \cite{BDR}}. Let $X$ be a finite CW complex and let $R$ be a connective ring spectrum.  Let 
\[
\mGL_{n}R = \mathbf{Aut}_{R}(R^{\vee n})
\]
be the derived mapping space of homotopy automorphisms of the $n$-fold wedge sum $R^{\vee n}$ with the topological monoid structure coming from composition of maps.  By Theorem \ref{thm:classification_of_Rbundles}, the classifying space $B \mGL_{n}R$ classifies stable equivalence classes of $R$-bundles with fiber $R^{\vee n}$.  Let $B \mGL_{\infty} R = \hocolim_{n} B \mGL_n R$.  Recall the following description of the infinite loop space underlying the algebraic $K$-theory spectrum of $R$:
\[
\Omega^{\infty} K(R) \simeq K_0 R \times B \mGL_{\infty} R^{+}
\]
The group $K_0 R = K^{f}_0 \pi_0 R$ is the algebraic $K$-theory of free $\pi_0 R$-modules, and the plus denotes Quillen's plus construction with respect to the commutator subgroup of $\pi_1 B \mGL_{\infty} R$.  Since the plus construction changes the homotopy type in general, we will need to work with lifted bundles, in the following sense.

\begin{definition}
A lifted $R$-bundle over $X$ is the data of:
\begin{itemize}
\item[(i)] An $H_*$-acyclic fibration $p \colon Y \arr X$ of CW complexes, by which we mean a Serre fibration with $\widetilde{H}_{*}(\mathrm{fiber}(p);\bZ) = 0$.
\item[(ii)] An $R$-bundle $E$ over $Y$.
\end{itemize}
We say that a lifted $R$-bundle $(E, Y, p)$ over $X$ is free if every fiber of $E$ admits a stable equivalence of $R$-modules $E_{y} \simeq R^{\vee n}$ for some $n$.
\end{definition}
Define a relation on lifted $R$-bundles over $X$ by declaring $(E, Y, p) \sim (E', Y', p')$ if there exists a map $f \colon Y \arr Y'$ over $X$ such that the induced map of $R$-modules $E \arr f^*E'$ over $Y$ is a stable equivalence.  This does \emph{not} define an equivalence relation in general, so we will work with the equivalence relation on lifted $R$-bundles over $E$ generated by $\sim$.  When convenient, we make the abbreviation $(E, Y) = (E, Y, p)$.

We assume from now on that $X$ is a finite CW complex.  Let $\Phi_{R}(X)$ be the set of equivalence classes of lifted free $R$-bundles over $X$.  The set $\Phi_{R}(X)$ is an abelian monoid under the operation $(E_1, Y_1) \oplus (E_2, Y_2)$ taking a pair of lifted $R$-bundles over $X$ to the lifted $R$-bundle 
\[
(g_1^*E_1 \vee_{Z} g_2^*E_2, Z), \quad \text{where $Z$ is the pullback} \quad \xymatrix{
Z \ar[r]^{g_2} \ar[d]_{g_1} & Y_2 \ar[d] \\
Y_1 \ar[r] & X }
\]
The zero of $\Phi_{R}(X)$ is the trivial $R$-bundle $(\ast_{X}, X)$ over $X$.  Let $\overline{K}_{R}(X)$ be the Grothendieck group of the monoid $\Phi_{R}(X)$.

We say that a lifted $R$-bundle is virtually trivial if there exists a space $T$ such that $\widetilde{H}_*(T ; \bZ) = 0$ and a map $f \colon Y \arr T$ (not necessarily over $X$) along with an $R$-bundle $(E', T)$ over $T$ and a stable equivalence of $R$-bundles $E \simeq f^*E'$.

\begin{lemma}\label{virtual_inverse_lemma}
Let $(E_1, Y_1)$ be a lifted free $R$-bundle over $X$.  Then there exists a lifted free $R$-bundle $(E_2, Y_2)$ over $X$ such that $(E_1, Y_1) \oplus(E_2, Y_2)$ is virtually trivial.
\end{lemma}
\begin{proof}
Let $f_1 \colon Y_1 \arr B \mGL_n R$ be a classifying map for $E_1$.  Let $P$ be the homotopy fiber of the $H_*$-acyclic fibration $Y_1 \arr  X$.  By \cite{haus_hus}*{1.3}, the kernel of $\pi_1 Y_1 \arr \pi_1 X$ is the perfect normal subgroup $\im (\pi_1 P \arr \pi_1 Y_1)$.  This is annihilated by the following map to the plus construction:
\[
\pi_1 P \arr \pi_1 Y_1 \overset{f_1}{\arr} \pi_1 B\mGL_nR \arr \pi_1 B\mGL_nR^{+}.
\]
By \cite{haus_hus}*{3.1}, $f_1$ descends to a map $g_1 \colon X \arr  B\mGL_nR^{+}$.  Use the grouplike $H$-space structure on $B \mGL_{\infty} R^{+}$ to find $g_2 \colon X \arr  B\mGL_m R^{+}$ such that $g_1 \oplus g_2 \colon X \arr B \mGL_{m + n} R^{+}$ is nullhomotopic.  Define $Y_2$ as the following pullback:
\[
\xymatrix{
Y_2 \ar[r]^-{f_2} \ar[d] & B \mGL_m R \ar[d] \\
X \ar[r]^-{g_2} & B \mGL_m R^{+} }
\]
We choose a model for the plus construction such that the right vertical map (and thus the left vertical map) is a $q$-fibration of CW complexes.  Let $E_2$ be the free $R$-bundle over $Y_2$ classified by the map $f_2$.  The sum $(E_1, Y_1) \oplus (E_2, Y_2)$ is a lifted $R$-bundle over the pullback $Y = Y_1 \times_{X} Y_2$ that is classified by a lift $f \colon Y \arr B \mGL_{m + n} R$ of $g_1 \oplus g_2$.  Thus $f$ is nullhomotopic, so it factors through the $H_*$-acyclic fiber of $B \mGL_{m + n} R \arr B \mGL_{m + n} R^{+}$, proving that $(E_1, Y_1) \oplus (E_2, Y_2)$ is virtually trivial.
\end{proof}

Given any space $X$, we generically write $r \colon X \arr \ast$ for the canonical map to a point, so that the pullback $r^* M$ is the trivially twisted $R$-bundle with fiber $M$.

\begin{lemma}
Suppose that $(E, Y)$ is a virtually trivial lifted $R$-bundle over $X$.  Then there exists a lifted $R$-bundle $(r^* M, Y')$ over $X$ that is equivalent to $(E, Y)$ as a lifted $R$-bundle: $[(E, Y)] = [(r^* M, Y')]$ in $\Phi_{R}(X)$.  If $E$ is a free $R$-bundle, then $M = R^{\vee n}$ for some $n$.
\end{lemma}
\begin{proof}
We are given an $H_*$-acyclic fibration $p \colon Y \arr X$, a map $f \colon Y \arr T$ where $\widetilde{H}_{*}(T) = 0$ and a stable equivalence $E \simeq f^* E'$ where $E'$ is an $R$-bundle over $T$.  Choose a point $t \colon * \arr T$.  Consider the following commutative diagram: 
\[
\xymatrix{
 & Y \ar[dl]_{p} \ar[d]_{g} \ar[dr]^{f} & \\
X &  T \times X \ar[l]_{\pi_2} \ar[r]^{\pi_1} & T \\
 & X \ar[ul]^{\id} \ar[u]^{\chi} \ar[r]_{r} & \ast \ar[u]_{t} }
\]
where $g(y) = (f(y), p(y))$ and $\chi(x) = (t, x)$.  The maps $p, \pi_2$ and $\id$ are all $H_*$-acyclic fibrations.  Form the $R$-bundle $\pi_1^* E'$ over $T \times X$.  Then we have a stable equivalence of $R$-bundles $g^* \pi_1^* E' = f^* E' \simeq E$ over $Y$.  On the other hand $\chi^* \pi_1^* E' \cong r^* t^* E'$ is a trivial bundle over $X$ with fiber $M = t^* E'$, since $t \circ r$ factors through a point.  The two triangles on the left show that $(E, Y) \sim (\pi_1^* E', T \times X)$ and $(\pi_1^* E', T \times X) \sim (r^* M, X)$.
\end{proof}

Consider the abelian group $K_0(R)$ as a discrete set and let $[X, K_0(R)]$ be the set of homotopy classes of maps from $X$, considered as an abelian group under the pointwise addition in the abelian group $K_0(R)$. Let $\psi \colon \Phi_{R}(X) \arr [X, K_0(R)]$ be the function that takes a lifted free $R$-bundle $(E, Y, p)$ to the map sending $x \in X$ to the equivalence class of the free $R$-module $E_x = (p \circ i)^*E$, where $i \colon \ast \arr Y$ is a choice of a point lying in the fiber of $p$ over $x$.  Since the fibers of $p$ are path-connected, different choices give the same equivalence class in $K_0(R)$ and it is easy to see that the definition depends only on the equivalence class of the lifted free $R$-bundle.  Since the abelian group structure on $K_0(R)$ is induced by the wedge sum of free $R$-modules, the function $\psi$ is a monoid homomorphism.  Let $\overline{\psi} \colon \overline{K}_R(X) \arr [X, K_0(R)]$ be the extension of $\psi$ to the Grothendieck group.  There is a natural splitting 
\[
\overline{K}_{R}(X) \cong \ker \overline{\psi} \oplus [X, K_0(R)]
\]
induced by the section of $\psi$ that takes an equivalence class $[R^{\vee n}] \in K_0(R)$ indexed by a path-component of $X$ to the trivially twisted $R$-bundle $r^* R^{\vee n}$ over that component.
Let $\Phi_R^n(X)$ be the set of equivalence classes of lifted $R$-bundles of rank $n$.  

\begin{proposition}  There is a natural isomorphism
\[
\ker \overline{\psi} \cong \colim_{n} \Phi_{R}^{n}(X).
\]
\end{proposition}
\begin{proof}
Suppose that $[E] - [F]$ is a formal difference of lifted free $R$-bundles in $\ker \overline{\psi}$.  We associate to $[E] - [F]$ the element $[E \oplus F'] \in \colim_{n} \Phi_{R}^{n}(X)$ where $F'$ is a lifted free $R$-bundle such that $F \oplus F'$ is virtually trivial (Lemma \ref{virtual_inverse_lemma}).  Conversely, to a class $[E] \in \Phi^{n}_{R}(X)$ we associate the formal difference $[E] - [T_n] \in \ker \overline{\psi}$, where $T_n = r^* R^{\vee n}$ is the trivial $R$-bundle of rank $n$. 
\end{proof}

\begin{proposition}  There is a natural isomorphism
\[
\colim_{n} \Phi_{R}^{n}(X) \cong [X, B\mGL_{\infty}(R)^{+}].
\]
\end{proposition}
\begin{proof}
Given the class of a lifted free $R$-bundle $(E, Y)$ over $X$ in $\colim_{n} \Phi_{R}^{n}(X)$, the arguments of Lemma \ref{virtual_inverse_lemma} show that the classifying map $f$ of $E$ extends to a map $g$ from $X$ to the plus construction:
\[
\xymatrix{
Y \ar[r]^-{f} \ar[d]_{p} & B \mGL_{n} R \ar[d] \\
X \ar[r]^-{g} & B \mGL_n R^{+} }
\]
Conversely, given a classifying map $g$ define $Y$ as the pullback displayed in the same diagram.   Then $p$ is an $H_*$-acyclic fibration and $f$ classifies a lifted free $R$-bundle $(E, Y)$ over $X$.  
\end{proof}

All together, we have proved:
\[
\overline{K}_{R}(X) \cong [X, K_0(R)] \oplus [X, B \mGL_{\infty}(R)^{+}]  \cong [X, \Omega^{\infty}K(R)].
\]
This completes the proof of Theorem \ref{main_theorem}.

\end{document}